\documentclass[12pt]{article}
\usepackage{setspace}
\usepackage{amsmath,amsthm,amssymb,mathtools}
\usepackage{dsfont}
\usepackage{array}
\usepackage{ytableau}
\ytableausetup{boxsize=1.55em,aligntableaux=center}

\usepackage{graphicx}
\usepackage{tikz}
\usepackage{tikz-cd}
\usetikzlibrary{calc}

\usepackage[margin=1in]{geometry}
\usepackage[dvipsnames]{xcolor}
\usepackage[hidelinks]{hyperref}
\usepackage{enumitem}

\newcommand{\Gr}{\operatorname{Gr}}
\newcommand{\Aut}{\operatorname{Aut}}

\newtheorem{theorem}{Theorem}
\newtheorem{corollary}[theorem]{Corollary}
\newtheorem{lemma}[theorem]{Lemma}
\newtheorem{prop}[theorem]{Proposition}

\theoremstyle{definition}
\newtheorem{definition}[theorem]{Definition}

\begin{document}

\title{\vspace{-0.5em}\Large There are no unexpected toric positroid varieties}
\author{\normalsize Avery McNeill}
\date{}

\maketitle
\vspace{-2.7em}

\begingroup
\renewcommand{\thefootnote}{}
\footnotetext{Date: October 5, 2026.

The author was partially supported by NSF DMS-1945212.}
\endgroup

\begin{abstract} Let $\mathcal{M}$ be a positroid. We prove that the positroid variety $\Pi_\mathcal{M}$ is toric with respect to some algebraic torus if and only if $\mathcal{M}$ is regular. In that case, $\Pi_\mathcal{M}$ is toric with respect to the standard diagonal torus. Thus, there are no unexpected toric positroid varieties.

\end{abstract}

\vspace{.4\baselineskip}
\noindent\hspace*{1em}\begin{minipage}{\dimexpr\linewidth-2em\relax}
\small\noindent\textbf{Note added in revision.} \textit{After this paper was posted to arXiv, the author learned that Gonzales, Le, and Lowen established Theorem 1 in \cite[Theorem~B.5]{GLL26}, using a similar approach. The proof presented here was developed independently of their work.}
\end{minipage}\par\vspace{.7\baselineskip}

Richardson varieties in the flag variety and positroid varieties in
the Grassmannian are closely related: in type $A$, positroid varieties
are projected Richardson varieties; see \cite{KLS14} and
\cite{Speyer23}, and see also \cite{BC12} on
projection varieties. Recall that an algebraic torus is an algebraic group
isomorphic to $(\mathbb C^\times)^d$ for some $d\geq 0$, and that a
toric variety is a normal irreducible variety admitting an effective
algebraic torus action with a dense open orbit. More generally, if an
algebraic torus $T$ acts on an irreducible variety $X$, we say that
$X$ is \emph{toric with respect to $T$} if the image of
$T$ in $\Aut(X)$ has a dense open orbit. The classification of toric Richardson varieties has recently attracted further attention.

Richardson varieties that are toric with respect to the standard torus have been characterized in type $A$ by Tsukerman and Williams \cite[Proposition~4.12]{TW15} and in arbitrary type by Can and Saha \cite[Theorem~4.7]{CanSaha}. More recently, Gorsky, Kim, and Sherman-Bennett \cite[Theorem~1.2]{GKSB26} characterized, in type $A,$ the Richardson varieties that are toric with respect to an arbitrary algebraic torus: these are precisely the Richardson varieties for which the associated Bruhat interval $[v,w]$ contains no subinterval isomorphic to $S_3.$ They call a toric Richardson \textit{unexpected} when its dimension exceeds $n-1.$ Such a variety is not toric with respect to the standard diagonal torus. They also proved that every projection of such a Richardson variety to the associated positroid variety in the Grassmannian is toric with respect to the standard torus \cite[Theorem~3.15]{GKSB26}. Thus one could say ``unexpected toric Richardson varieties project to expected toric positroid varieties." This raises the question: ``Are there any unexpected toric positroid varieties?"

We answer this question in the negative. More precisely, a positroid variety is toric with respect to some algebraic torus if and only if its indexing positroid is regular; in that case, it is already toric with respect to the standard diagonal torus. Our main result is the following:
\enlargethispage{3\baselineskip}

\begin{theorem}\label{thm:main}
Let $\mathcal{M}$ be a positroid and let
$\Pi_{\mathcal{M}}\subseteq \Gr(k,n)$ be the corresponding positroid
variety over $\mathbb{C}.$ The following are equivalent:
\begin{enumerate}[topsep=0pt,itemsep=0pt,parsep=0pt]
\item $\Pi_\mathcal{M}$ is toric with respect to some algebraic torus.
\item $\mathcal{M}$ is a regular matroid.
\item $\Pi_{\mathcal{M}}$ is toric with respect to the standard diagonal
torus $T_{\mathrm{std}}.$
\end{enumerate}
\end{theorem}

Note that toricity of positroid varieties with respect to the standard torus admits some characterizations, including a moment-polytope criterion given by Tsukerman and Williams \cite[Proposition~7.12]{TW15} and the equivalent plabic-forest criterion of {\L}ukowski, Parisi, and Williams \cite[Propositions~3.15--3.16]{LPW23}.

Throughout, $\operatorname{Gr}(k,n)$ denotes the complex Grassmannian, with its usual
$\mathrm{GL}_n(\mathbb{C})$-action. Since scalar matrices act trivially, this action
factors through $\mathrm{PGL}_n(\mathbb{C})$; we write
\[
T_{\mathrm{std}}=(\mathbb{C}^{\times})^n/\mathbb{C}^{\times}
\]
for the image of the diagonal torus, and, when a positroid variety
$X\subseteq\operatorname{Gr}(k,n)$ is fixed, also for the image of its restricted
action in $\operatorname{Aut}(X)$.

\begin{definition}[Positroids, positroid varieties, and regularity]
Let $\mathcal M$ be a rank-$k$ matroid on $[n]$, and write
$\mathcal B(\mathcal M)\subseteq\binom{[n]}k$ for its set of
bases. For $I\in\binom{[n]}k$, let $p_I$ denote the
corresponding Pl\"ucker coordinate. For a field $\mathbb F$ and
a point $V\in\Gr_{\mathbb F}(k,n)$, let $\mathcal M_V$ be the
rank-$k$ matroid whose bases are the sets $I$ for which
$p_I(V)\neq0$.

A rank-$k$ matroid $\mathcal M$ on $[n]$ is a
\emph{positroid} if $\mathcal M=\mathcal M_V$ for some
$V\in\Gr_{\mathbb R}(k,n)$ represented by a full-rank real matrix
whose maximal minors are all nonnegative. For a positroid
$\mathcal M$, its \emph{positroid variety} is
\[
\Pi_{\mathcal M}
=
V\bigl(p_I:I\notin\mathcal B(\mathcal M)\bigr)
\subseteq\Gr(k,n).
\]
By \cite[Theorem~5.15]{KLS13}, these Pl\"ucker coordinates generate
the scheme-theoretic defining ideal of $\Pi_{\mathcal M}$. We write
$\Pi_{\mathcal M}^{\circ}$ for the open positroid variety indexed by
$\mathcal M$, whose closure is $\Pi_{\mathcal M}$. A matroid is \emph{regular} if it is representable over every field. \end{definition}

The proof of Theorem~\ref{thm:main} proceeds as follows. We first conjugate an arbitrary acting torus so that it contains the image of the standard diagonal torus. We then show that a nonregular positroid gives rise to a closed, torus-stable positroid subvariety isomorphic to $\Gr(2,4)$, which shows that the corresponding positroid variety cannot be toric with respect to any algebraic torus. Finally, we use projective uniqueness of representations of regular matroids to prove that every regular positroid variety is toric with respect to $T_{\mathrm{std}}$.

We begin by placing the acting torus in a convenient position.

\begin{lemma}\label{lem:normalization}
Let $X=\Pi_{\mathcal M}$ be toric with respect to an algebraic torus
$H$. Then the image of $H$ in $\Aut(X)$ is conjugate in $\Aut^0(X)$ to a torus containing
$T_{\mathrm{std}}$.

\end{lemma}

\begin{proof}
Replacing $H$ by its image in $\Aut(X)$, we may assume that
$H$ acts effectively. Since $H$ is connected, it is contained in
$\Aut^0(X)$. By \cite[Theorem~5.13]{KLS13}, $X$ is normal, and
$X$ is projective because it is closed in the Grassmannian. Since
$X$ is toric with respect to $H$, it is a complete toric variety.

By \cite[Proposition~11]{Demazure}, as extended to arbitrary complete toric varieties in \cite[Theorem~2]{LLA}, $\Aut^0(X)$ is a connected
linear algebraic group. Further, $H$ is a maximal torus of
$\Aut^0(X)$ by dimension. The standard diagonal torus preserves the
Pl\"ucker-coordinate ideal of $X$ by
\cite[Theorem~5.15]{KLS13}. Its image is connected, and hence
lies in $\Aut^0(X)$; we continue to denote this image by
$T_{\mathrm{std}}$.

Choose a maximal torus $L\subseteq\Aut^0(X)$ containing
$T_{\mathrm{std}}$. Since maximal tori in a connected linear
algebraic group are conjugate, $H$ and $L$ are conjugate. The
$L$-action therefore also has a dense open orbit.
\end{proof}

We now turn to the matroid-theoretic input for the obstruction. For a rank-$k$ matroid $\mathcal M$ on $[n]$ and a field
$\mathbb F$, its \emph{matroid stratum} is
\[
S_{\mathcal M}(\mathbb F)
=
\{V\in\Gr_{\mathbb F}(k,n):\mathcal M_V=\mathcal M\}.
\]
We use the following result of Quail and Rombach.

\begin{theorem}[Quail--Rombach
{\cite[Corollary~4.1 and Theorem~4.3]{QR}}]
\label{thm:quail-rombach}
Let $\mathcal P$ be a positroid.
\begin{enumerate}[label=\roman*)]
\item $\mathcal P$ is regular if and only if it is binary.
\item If $\mathcal P$ is regular, then
\[
\Pi_{\mathcal P}^{\circ}(\mathbb F)
=
S_{\mathcal P}(\mathbb F)
\]
for every field $\mathbb F$.
\end{enumerate}
\end{theorem}

In particular, the second item implies that if $\mathcal P$ is
regular and
\[
V_1,V_2\in
\Pi_{\mathcal P}^{\circ}(\mathbb F)
\subseteq \Gr_{\mathbb F}(k,n),
\]
then any full-rank matrices representing $V_1$ and $V_2$
both represent the matroid $\mathcal P$ over $\mathbb{F}$.

We next lift positroid minors to closed positroid subvarieties.
For $A\subseteq E$, write
\[
\mathbb C^A=\operatorname{span}\{e_i:i\in A\},
\qquad
e_A=\sum_{i\in A}e_i\in\mathbb R^E.
\]
For a matroid $\mathcal M$ on $E$, its base polytope is
\[
\Gamma_{\mathcal M}
=
\operatorname{conv}
\{e_B:B\in\mathcal B(\mathcal M)\}
\subseteq\mathbb R^E.
\]
\begin{lemma}\label{lem:minor-lift}
Let $\mathcal{M}$ be a rank-$k$ positroid on $E=[n].$ Let
$\mathcal{Q}$ be a rank-$r$ minor of $\mathcal{M},$ realized on a ground set $S\subseteq E$. Then there exist
a subset $C\subseteq E\setminus S$ and a positroid $\hat{\mathcal{Q}}$
on $E$ such that
the following hold:
\begin{enumerate}[label=\roman*)]
\item $|C|= k-r$;
\item viewing $\mathbb{R}^S$ as the corresponding coordinate subspace of $\mathbb{R}^E,$ $\Gamma_{\hat{\mathcal{Q}}}=e_C+\Gamma_{\mathcal{Q}}$ is a face
of $\Gamma_{\mathcal{M}}$;
\item $\Pi_{\hat{\mathcal{Q}}}\subseteq \Pi_{\mathcal{M}}$
scheme-theoretically;
\item the coordinate embedding
\begin{equation*}\begin{split}
\iota_C:\Gr(r,\mathbb{C}^S)&\longrightarrow \Gr(k,\mathbb{C}^E),\\
W&\longmapsto \mathbb{C}^C\oplus W
\end{split}\end{equation*}

restricts to an isomorphism
$$
\Pi_{\mathcal{Q}}\xrightarrow{\sim}
\Pi_{\hat{\mathcal{Q}}}.
$$
\end{enumerate}
\end{lemma}

\begin{proof} Positroids are closed under taking minors  \cite[Proposition~3.5]{ARW16}, so $\mathcal{Q}$ is a positroid. By the Scum Theorem \cite[Theorem~3.3.1]{Oxley}, there is a subset $Z\subseteq E\setminus S$ such that $(\mathcal{M}/Z)|_S = \mathcal{Q}$ and $r(\mathcal{M}/Z) = r.$ Let $C$ be a basis of $Z$ in $\mathcal{M}.$ Then $C$ is independent and, by the contraction-rank formula \cite[Proposition~3.1.6]{Oxley}, $|C| = k-r.$ The elements of $Z\setminus C$ are loops in $M/C,$ so restricting to $S$ gives $$(M/C)|_S=\mathcal{Q}_.$$

Let $D = E\setminus(C\sqcup S),$ and define
\begin{equation*}\begin{split}
\omega:\mathbb{R}^{E}&\longrightarrow\mathbb{R},\\
e_i&\longmapsto 0 \text{ if } i\in C,\\
e_i&\longmapsto 1 \text{ if } i\in S,\\
e_i&\longmapsto 2 \text{ if } i\in D.
\end{split}
\end{equation*}
For a basis $B,$ $\omega(e_B)\geq r$ and
equality holds exactly when $B=C\sqcup I$, where $I\subseteq S$ is
a basis of $\mathcal{Q}.$ Indeed, if $c=|B\cap C|,$ $s=|B\cap S|,$ and
$d=|B\cap D|,$ then
$$\omega(e_B)=s+2d=k-c+d\geq k-|C|=r.$$ Equality holds exactly
when $c=|C|$ and $d=0$. In this case, $B=C\sqcup I$ for some $I\subseteq S,$ and $I$ is a basis of $(\mathcal{M}/ C)|_{S}=\mathcal{Q}.$

Conversely, if $I$ is a basis of
$\mathcal Q=(\mathcal M/C)|_S$, then $C\sqcup I$ is a basis of
$\mathcal M$, and hence $\omega(e_{C\sqcup I})=r$. Thus, the minimizing face of $\Gamma_\mathcal{M}$ is $$F=e_C+\Gamma_\mathcal{Q}.$$ By \cite[Corollary~5.4]{ARW16}, $F$ is the base polytope of a positroid $\hat{\mathcal{Q}}$ on $E,$ so we can call it $\Gamma_{\hat{\mathcal{Q}}},$ with
$$
\mathcal{B}(\hat{\mathcal{Q}})=
\{C\sqcup I:I\in\mathcal{B}(\mathcal{Q})\}.
$$
Since $\mathcal{B}(\hat{\mathcal{Q}})\subseteq \mathcal{B}(\mathcal{M}),$ every nonbasis of $\mathcal{M}$ is a nonbasis of $\hat{\mathcal{Q}}.$ Therefore, Theorem 5.15 of \cite{KLS13} gives
$$
I(\Pi_{\mathcal M})\subseteq I(\Pi_{\widehat{\mathcal Q}}),
$$
and hence $\Pi_{\widehat{\mathcal Q}}\subseteq\Pi_{\mathcal M}$ scheme-theoretically. 

Finally, the equations $$p_J=0 \,\,\,\,\,\, \text{ whenever }C\nsubseteq J\text{ or }J\cap D\neq\emptyset$$ cut out the image of \begin{equation*}\begin{split}\iota_C:\Gr(r,\mathbb{C}^S)&\rightarrow\Gr(k,\mathbb{C}^E)\\ W&\mapsto\mathbb{C}^C\oplus W.\end{split}\end{equation*} Indeed, these equations say that the corresponding $k$-plane
contains $\mathbb C^C$ and is contained in
$\mathbb C^{C\sqcup S}$. On this image, $$
p_{C\sqcup I}(\iota_C(W))=\pm p_I(W)
\quad\text{for }I\in\binom{S}{r},
$$ where the sign comes from reordering the elements of $C,$ followed by those of $I,$ into the fixed order on $(C\sqcup I).$ The remaining nonbasis Pl\"{u}cker equations for $\Pi_{\hat{\mathcal{Q}}}$ therefore pull back exactly to the nonbasis Pl\"{u}cker equations for $\Pi_\mathcal{Q}.$ Thus $\iota_C$ restricts to an isomorphism $$
\Pi_{\mathcal Q}\xrightarrow{\sim}\Pi_{\widehat{\mathcal Q}}.
$$
\end{proof}

Combining Lemma~\ref{lem:minor-lift} with Theorem~\ref{thm:quail-rombach} yields the following corollary.

\begin{corollary}\label{cor:gr24}
If $\mathcal M$ is a nonregular positroid, then
$\Pi_{\mathcal M}$ contains a closed subvariety isomorphic to
$\Gr(2,4)$.
\end{corollary}

\begin{proof}
Since $\mathcal{M}$ is nonregular, it is not binary by
Theorem~\ref{thm:quail-rombach}. A matroid is binary if
and only if it has no $U_{2,4}$-minor
\cite[Theorem~6.5.4]{Oxley}. Thus, $\mathcal{M}$ has a $U_{2,4}$-minor.

Lemma~\ref{lem:minor-lift} implies that $\Pi_{\mathcal{M}}$ contains
a closed positroid subvariety isomorphic to
$\Pi_{U_{2,4}}.$ Since $U_{2,4}$ has no nonbases, $$\Pi_{U_{2,4}} = \Gr(2,4).$$ 
\end{proof}

We next show that, after applying Lemma~\ref{lem:normalization}, every closed
positroid subvariety is stable under the acting torus. For each
$a\in[n]$, let $F_\bullet^{(a)}$ be the cyclically shifted coordinate
flag
$$
0\subset \langle e_{\overline{a}}\rangle
 \subset \langle e_{\overline{a}},e_{\overline{a+1}}\rangle
 \subset \cdots
 \subset \langle e_{\overline{a}},e_{\overline{a+1}},\ldots,e_{\overline{a+n-1}}\rangle
 =\mathbb C^n,
$$
where this $\overline{a+j}$ notation simply means that the indices are taken modulo $n$. Choose a diagonal
one-parameter subgroup
$$
\lambda_a:\mathbb C^\times\longrightarrow T_{\mathrm{std}}
$$
that is strongly dominant for the Borel subgroup determined by
$F_\bullet^{(a)}$. Here strongly dominant means that its pairing
with every positive root is positive; such a cocharacter always
exists, and in type $A$ its diagonal weights may be chosen
strictly ordered compatibly with $F_\bullet^{(a)}$, so choosing such a subgroup poses no problem.

For $I\in\binom{[n]}{k}$, let
$V_I=\langle e_i:i\in I\rangle$. The \textit{Bia{\l}ynicki-Birula attracting cell}
$$
A_I^{(a)}
=
\left\{
V\in\operatorname{Gr}(k,n):
\lim_{t\to0}\lambda_a(t)V=V_I
\right\}
$$
is the Schubert cell through $V_I$ with respect to
$F_\bullet^{(a)}$; see \cite[Lemma~6.1]{BuchChaputPerrin}. The cyclically shifted
case follows from the standard case by conjugating by the
corresponding cyclic permutation matrix. By
\cite[Lemma~5.3]{KLS13}, every open positroid variety is an
intersection of such cells.

The following lemma is analogous to \cite[Theorem~3.1]{CanSaha}.

\begin{lemma}\label{lem:stability}
Let $X=\Pi_{\mathcal{M}}\subseteq \Gr(k,n)$ be a positroid
variety. Let $H\subseteq\Aut^0(X)$ be a torus containing
$T_{\mathrm{std}}.$ Then every closed positroid subvariety
$\Pi_{\mathcal{N}}\subseteq\Pi_{\mathcal{M}}$ is $H$-stable.
\end{lemma}

\begin{proof}
Since $H$ is a torus, it is abelian and therefore centralizes
the image of each $\lambda_a$. The fixed-point set
$X^{\lambda_a}$ is finite, since
\[
X^{\lambda_a}\subseteq\Gr(k,n)^{\lambda_a}
\]
and $\Gr(k,n)^{\lambda_a}$ is the finite set of coordinate
$k$-planes. Since $H$ is connected and preserves
$X^{\lambda_a}$, it fixes $X^{\lambda_a}$ pointwise.

Fix $a\in[n]$ and $I\in\binom{[n]}{k}$ such that the cell
$A^{(a)}_I$ occurs in the intersection that defines
$\Pi_{\mathcal{N}}^\circ,$ as given in \cite[Lemma~5.3]{KLS13}. Let
$$
X^+_{I,a}=X\cap A^{(a)}_I.
$$
If $x\in X^+_{I,a},$ then
$\lim_{t\rightarrow 0}\lambda_a(t)x=V_I.$ Further, since $\lambda_a(\mathbb{C}^\times) \subseteq T_{\text{std}},$ we know that $\lambda_a(t)x\in X$ for all $t\neq 0.$ Since $X$ is closed, $V_I\in X.$ Moreover, as $V_I$ is a coordinate $k$-plane, it is fixed by $\lambda_a;$ thus, $V_I\in X^{\lambda_a}$. So, if we let $g\in H,$ $gV_I=V_I.$ Therefore,
\begin{equation*}\begin{split}
\lim_{t\rightarrow 0}\lambda_a(t)(gx)&=g\left(\lim_{t\rightarrow 0}\lambda_a(t)x\right)\\
&=gV_I\\
&=V_I.
\end{split}\end{equation*}

Thus, $gx\in X^+_{I,a}$, so $X^+_{I,a}$ is $H$-stable.
Each corresponding intersection $X\cap A_I^{(a)}$ is therefore
$H$-stable. Since $\Pi_{\mathcal N}^{\circ}$ is the intersection
of these sets, it is $H$-stable, and so is its closure
$\Pi_{\mathcal N}$.

\end{proof}

Combining Lemma~\ref{lem:normalization},
Corollary~\ref{cor:gr24}, and Lemma~\ref{lem:stability}
gives the following corollary.

\begin{corollary}\label{cor:regular}
If the positroid variety $\Pi_\mathcal{M}$ is toric with respect to some algebraic torus, then $\mathcal{M}$ is regular.
\end{corollary}

\begin{proof}
Let $H$ be an algebraic torus for which $\Pi_{\mathcal M}$ is
toric. Replacing $H$ by its image in $\Aut(\Pi_{\mathcal{M}}),$ and then by a conjugate as in Lemma~\ref{lem:normalization}, we may assume that
$T_{\mathrm{std}}\subseteq H$.

Suppose that $\mathcal M$ is not regular. By
Corollary~\ref{cor:gr24}, $\Pi_{\mathcal M}$ contains a closed
positroid subvariety
\[
Y\cong\Gr(2,4).
\]
By Lemma~\ref{lem:stability}, $Y$ is $H$-stable. Every
irreducible closed $H$-stable subvariety of a toric $H$-variety
is toric with respect to the effective quotient of $H$ acting
on it; see \cite[Section~3.1]{Fulton}. Thus $Y$ is toric.

This is impossible, as an effective torus with a dense orbit on
$\Gr(2,4)$ would have dimension
\[
\dim\Gr(2,4)=4,
\]
whereas
\[
\Aut^0(\Gr(2,4))\cong \mathrm{PGL}_4(\mathbb{C})
\] has maximal tori of dimension $3$ \cite[Theorem~1.1]{Cowen}.
\end{proof}

We now use projective uniqueness of representations of regular matroids in order to show that regularity is equivalent to toricity with respect to the standard torus. 

Since every regular matroid is binary, {\cite[Proposition~6.6.5]{Oxley}} gives the following. 
\begin{prop}
\label{prop:projective-equivalence}
Let $\mathcal M$ be a regular matroid and let $\mathbb F$ be a
field. Then any two full-rank matrix representations of
$\mathcal M$ over $\mathbb F$ are projectively equivalent.
That is, one can be obtained from the other by row operations
and nonzero column scaling.
\end{prop}

Together with Corollary~\ref{cor:regular}, Theorem~\ref{thm:quail-rombach} and
Proposition~\ref{prop:projective-equivalence} give the following characterization.

\begin{theorem}\label{thm:standard-toric}
The positroid variety $\Pi_{\mathcal M}$ is toric with respect
to the standard torus if and only if $\mathcal M$ is regular.
\end{theorem}

\begin{proof}
The forward implication is Corollary~\ref{cor:regular}.

For the reverse direction, suppose that $\mathcal{M}$ is regular. By Theorem~\ref{thm:quail-rombach},
$$\Pi_{\mathcal{M}}^\circ(\mathbb{C})=S_{\mathcal{M}}(\mathbb{C}).$$ In other words, every full-rank matrix representing some $V\in\Pi_{\mathcal{M}}^\circ(\mathbb{C})$ represents
$\mathcal{M}.$ Choose $V_0,V\in\Pi_{\mathcal M}^{\circ}(\mathbb C)$, and let
$A_0$ and $A$ be full-rank $k\times n$ matrices whose row spaces are
$V_0$ and $V$, respectively. Both matrices represent the same matroid
$\mathcal M$ over $\mathbb C$.

By Proposition~\ref{prop:projective-equivalence}, there exist
$g\in \mathrm{GL}_k(\mathbb{C})$ and a nonsingular diagonal matrix $d$
such that $A=gA_0d.$ Passing to row spaces gives
$V=d\cdot V_0.$ Thus,
$$\Pi_{\mathcal{M}}^\circ(\mathbb{C})\subseteq T_{\mathrm{std}}\cdot V_0.$$ Conversely, diagonal rescaling preserves which Pl\"ucker coordinates vanish, so
$$
T_{\mathrm{std}}\cdot V_0\subseteq S_{\mathcal M}(\mathbb C)=\Pi_{\mathcal M}^{\circ}(\mathbb{C}).
$$
Hence
$$
\Pi_{\mathcal M}^{\circ}(\mathbb{C})=T_{\mathrm{std}}\cdot V_0.
$$
Since $\Pi_{\mathcal M}^{\circ}$ is dense in $\Pi_{\mathcal M}$, the effective image of $T_{\mathrm{std}}$ has a dense open orbit in $\Pi_\mathcal{M}.$ Thus $\Pi_\mathcal{M}$ is toric with respect to the standard torus. 
\end{proof}

\begin{proof}[Proof of Theorem~\ref{thm:main}]
If $\Pi_{\mathcal M}$ is toric with respect to some algebraic
torus, then $\mathcal M$ is regular by
Corollary~\ref{cor:regular}. Theorem~\ref{thm:standard-toric}
then shows that $\Pi_{\mathcal M}$ is toric with respect to
$T_{\mathrm{std}}$.
\end{proof}

\section*{Acknowledgments}

I would like to thank Dave Anderson for helpful conversations, both in directing me toward this question and during my work on it. I also thank Melissa Sherman-Bennett for comments on an early draft.

\section*{AI Usage Statement}

ChatGPT-5.5 was used to do literature search, assist in proving Lemma 7, and proofread.

\end{document}